\documentclass[a4paper,reqno,12pt]{amsart}

\usepackage{amsmath, amssymb, amsthm, enumerate,amsfonts,latexsym, mathrsfs}

\usepackage[top=30mm,right=27mm,bottom=30mm,left=27mm]{geometry}

\newtheorem{theorem}{Theorem}[section]
\newtheorem{corollary}[theorem]{Corollary}
\newtheorem{lemma}[theorem]{Lemma}
\newtheorem{proposition}[theorem]{Proposition}

\makeatletter
\newcommand{\imod}[1]{\allowbreak\mkern4mu({\operator@font mod}\,\,#1)}
\makeatother

\newcommand{\Irr}{{\mathrm {Irr}}}

\newcommand{\Out}{{\mathrm {Out}}}

\newcommand{\PSL}{{\mathrm {PSL}}}
\newcommand{\GL}{{\mathrm {GL}}}
\newcommand{\PSp}{{\mathrm {PSp}}}
\newcommand{\PSU}{{\mathrm {PSU}}}
\newcommand{\SL}{{\mathrm {SL}}}
\newcommand{\SU}{{\mathrm {SU}}}

\newcommand{\PGL}{{\mathrm {PGL}}}
\newcommand{\PGU}{{\mathrm {PGU}}}
\newcommand{\GU}{{\mathrm {GU}}}

\newcommand{\SSS}{\mathrm{S}}
\newcommand{\A}{\mathrm{A}}

\newcommand{\Chevie}{{\sf Chevie}}
\newcommand{\Magma}{{\sf Magma}}
\newcommand{\Atlas}{{\sf Atlas}}

\theoremstyle{definition}

\newtheorem{problem}{Problem}

\begin{document}
\title[\textbf{Zeros of primitive characters}]{\textbf{Zeros of Primitive Characters of Finite Groups}}

\author{Sesuai Yash Madanha}
\address{Sesuai Yash Madanha, Department of Mathematics and Applied Mathematics, University of Pretoria, Private Bag X20, Hatfield, Pretoria 0028, South Africa}
\address{DST-NRF Centre of Excellence in Mathematical and Statistical Sciences (CoE-MaSS)}
\email{madanhasy@gmail.com}

\thanks{The author acknowledges the support of DST-NRF Centre of Excellence in Mathematical and Statistical Sciences (CoE-MaSS). Opinions expressed and conclusions arrived at are those of the author and should not necessarily be attributed to the CoE-MaSS.}


\date{\today}


\begin{abstract}
We classify finite non-solvable groups with a faithful primitive complex irreducible character that vanishes on a unique conjugacy class. Our results answer a question of Dixon and Rahnamai Barghi and suggest an extension of Burnside's classical theorem on zeros of characters.
\end{abstract}

\maketitle


\section{Introduction}\label{s:intro}

This work is a continuation of what we began in \cite{Mad18q}, the classification of finite non-solvable groups with a faithful primitive complex irreducible character that vanishes on exactly one conjugacy class. In \cite{Mad18q}, we completed the work for groups with a non-abelian composition factor isomorphic to a sporadic simple group, an alternating group, $ \A _{n} $, $ n\geq 5 $ or a special linear group $ \PSL_{2}(q) $, $ q\geq 4 $. In this article we finish the classification for all finite non-solvable groups. This is a contribution to the more general problem of the classification of finite groups with an irreducible character that vanishes on exactly one conjugacy class (see \cite{Zhm79,DB07,Qia07, BT-V15, Mad18q} for more work on this problem). We begin by looking at the following problem:\\ 
\begin{problem}\label{classifyp-powerorder}
For each quasisimple group $M$, classify all faithful complex irreducible characters $\chi$ such that there exists some prime $ p $ such that 
\begin{itemize}
\item[(i)] $ \chi $ vanishes on elements of the same $p$-power order;
\item[(ii)] the number of conjugacy classes that $\chi$ vanishes on is at most the size of the outer automorphism group of the group $M/Z(M)$;
\item[(iii)] $ Z(M) $ is cyclic and of $ p $-power order.
\end{itemize}
\end{problem}
For convenience we shall define the following property for a finite group $ M $. We shall say \eqref{e:star} holds for $ M $ if:
\[\mbox{\emph{A faithful irreducible character $ \chi $ of $ M $ has properties (i)-(iii) of  Problem \ref{classifyp-powerorder}.} \label{e:star} \tag{$\star$}}\] 

We completely solve Problem \ref{classifyp-powerorder} in this article:

\begin{theorem}\label{classificationp-powerorder}
Let $ M $ be a quasisimple group. If \eqref{e:star} holds for $ M $, then $ M $ is one of the following:
\begin{itemize}
\item[(1)] $ M=\PSL_{2}(5) $, $ \chi (1)=3 $ or $ \chi (1)=4 $;
\item[(2)] $ M= \SL_{2}(5) $, $ \chi (1)=2 $ or $ \chi (1)=4 $;
\item[(3)] $ M=3{\cdot} \A_{6} $, $ \chi(1)=9 $;
\item[(4)] $ M= \PSL_{2}(7) $, $ \chi(1)=3 $;
\item[(5)] $ M= \PSL_{2}(8) $, $ \chi (1)=7 $;
\item[(6)] $ M=\PSL_{2}(11) $, $ \chi(1)=5 $ or $ \chi(1)=10 $;
\item[(7)] $ M= \PSL_{2}(q)$, $ \chi(1)=q $, where $ q\geq 5 $;
\item[(8)] $ M=\PSU_{3}(4) $, $ \chi(1)=13 $;
\item[(9)] $ M=$  $^{2}\mathrm{B}_{2}(8)$, $ \chi(1)=14 $.
\end{itemize}
\end{theorem}

Using Theorem \ref{classificationp-powerorder}, we classify finite non-solvable groups with a faithful primitive irreducible character that vanishes on one conjugacy class. We showed in \cite{Mad18q} that it is sufficient to only consider automorphism groups of the groups in Theorem \ref{classificationp-powerorder}.

\begin{theorem}\label{classificationoneclass}
Let $ G $ be a finite non-solvable group. Then $ \chi \in \Irr (G) $ is faithful, primitive and vanishes on one conjugacy class if and only if $ G $ is one of the following groups:
\begin{itemize}
\item[(1)] $ G = \PSL_{2}(5) $, $ \chi (1)=3 $ or $ \chi (1)=4 $;
\item[(2)] $ G= \SL_{2}(5) $, $ \chi (1)=2 $ or $ \chi (1)=4 $;
\item[(3)] $ G\in \{\A_{6}{:}2_{2},~ \A_{6}{:}2_{3},~ 3{\cdot} \A_{6}{:}2_{3}\} $, $ \chi(1)=9 $ for all such $ G $;
\item[(4)] $ G=\PSL_{2}(7) $, $ \chi (1)= 3 $;
\item[(5)] $ G=\PSL_{2}(8){:}3 $, $ \chi (1)=7 $;
\item[(6)] $ G=\PGL_{2}(q) $, $ \chi (1)=q $, where $ q\geq 5 $;
\item[(7)] $ G=$  $^{2}\rm{B_{2}}(8){:}3 $, $ \chi(1)=14 $.
\end{itemize}
\end{theorem} 

The result below follows easily:

\begin{corollary}\label{answertodixon}
Let $ G $ be a finite non-abelian simple group and let $ \chi \in \Irr(G) $. If $ \chi $ vanishes on exactly one conjugacy class, then one of the following holds:
\begin{itemize}
\item[(1)] $ G=\PSL_{2}(5) $, $ \chi(1)=3 $;
\item[(2)] $ G=\PSL_{2}(7) $, $ \chi(1)=3 $;
\item[(3)] $ G=\PSL_{2}(2^{a}) $, $ \chi(1)=2^{a} $, where $ a\geqslant 2 $.
\end{itemize}
\end{corollary}

Corollary \ref{answertodixon} positively answers a question posed by Dixon and Rahnamai Barghi \cite[Remark 11]{DB07}.

We now look at what our results imply with regards to the classical theorem of Burnside on zeros of characters. There have been some generalizations of Burnside's theorem (see \cite{MNO00}, \cite{BO04} and \cite{Nav01}).

Burnside's theorem can be rewritten as follows:

\begin{theorem} (\textbf{Burnside's Theorem})
Let $ G $ be a finite group and let $ \chi \in \Irr (G) $. If $ \chi(1) $ is divisible by a prime, then $ \chi $ vanishes on at least one conjugacy class. 
\end{theorem}

Using \cite[Propositions 1(i) and 4]{DB07} and Theorem \ref{classificationoneclass} we have the following theorem:
\begin{theorem}\label{burnside'sprimitivegeneralization}
Let $ G $ be a finite group whose non-abelian composition factors are not isomorphic to $ ^{2}\rm{B}_{2}(8) $. Let $ \chi \in \Irr(G) $ be primitive. If $ \chi(1) $ is divisible by two distinct prime numbers, then $ \chi $ vanishes on at least two conjugacy classes. 
\end{theorem}

The non-solvable group $ ^{2}\rm{B}_{2}(8){:}3 $ is a real exception by Theorem \ref{classificationoneclass}. Hence this extends Burnside's theorem when the character is primitive. It would be interesting to know if the primitivity is necessary for Theorem \ref{burnside'sprimitivegeneralization} to hold.

\section{Preliminaries}

In this section we present some results we will need to use. A lot of work on zeros of characters of quasisimple groups is found in \cite{MSW94,LM99,Mal99,MNO00,LM15,LMS16,LM16}. We shall use most of these results in this article. We need some definitions before we present the recent work of L\"ubeck and Malle \cite{LM16}. Let $ \Phi_{n} $ denote the $ n $-th cyclotomic  polynomial over $ \mathbb{Q} $. Let $ m, n $ be positive integers. Then by $ m|| n $, we mean that $ m|n $ but $ m^{2}\nmid n $. If $ l>2 $ not dividing $ q $, the multiplicative order of $ q $ modulo $ l $ is denoted by $ d_{l}(q) $.

\begin{theorem}\cite[Theorem 1]{LM16}\label{p-rank3}
Let $ l > 2 $ be a prime and $ M $ a finite quasisimple group of $ l $-rank at least $ 3 $. Then for any non-linear character $ \chi \in \Irr (M) $ there exists an $ l $-singular element $ g\in M $ with $ \chi (g)=0 $, unless either $ M $ is a finite group of Lie type in characteristic $ l $, or $ l=5 $ and one of the following hold:
\begin{itemize}
\item[(1)] $ M=\PSL_{5}(q) $ with $ 5||(q-1) $ and $ \chi $ is unipotent of degree $ \chi (1)=q^{2}\Phi _{5} $;
\item[(2)] $ M=\PSU_{5}(q) $ with $ 5||(q+1) $ and $ \chi $ is unipotent of degree $ \chi (1)=q^{2}\Phi _{10} $;
\item[(3)] $ M=Ly $ and $ \chi (1)\in \{ 48174, 11834746\} $; or
\item[(4)] $ M=\mathrm{E}_{8}(q) $ with $ q $ odd, $ d_{l}(q)=4 $ and $ \chi $ is one character in the Lusztig-series of type $ D_{8} $.
\end{itemize}
\end{theorem} 

Let $ \mathcal{M} $ be a simple, simply connected algebraic group over $ \mathbb{\overline{F}}_{p} $, the algebraic closure of a finite field of characteristic $ p $ and let $ F:\mathcal{M} \rightarrow \mathcal{M} $ be a Frobenius morphism such that $ M:=\mathcal{M}^{F} $, the finite group of fixed points. Let $ \mathcal{M}^{*} $ denote the dual group of $ \mathcal{M} $ with corresponding Frobenius morphism $ F^{*}: \mathcal{M}^{*}\rightarrow \mathcal{M}^{*} $. Then $ M^{*}:=(\mathcal{M}^{*})^{F^{*}} $ is the dual group of $ M $. Using Deligne-Lusztig theory, we have that irreducible characters of $ M $ are partitioned into Lusztig series $ \mathcal{E}(M,s^{*}) $ that are parametrised by conjugacy classes of semisimple elements $ s^{*} $ in the dual group $ M^{*} $. See \cite{Car85} and \cite{DM91} for basic results on Deligne-Lusztig theory of complex representations of finite groups of Lie type. 

The following lemma will be essential:

\begin{lemma} \cite[Lemma 3.2]{GM12}\label{vanishingresultTorus}
Let $ x\in M $ be semisimple and let $ \chi \in \mathcal{E}(M,s^{*}) $ be an irreducible character of $ M $ with $ \chi (x)\not= 0 $. Then there is a maximal torus $ T\leq M $ with $ x\in T $ such that $ T^{*}\leq \textbf{C}_{M^{*}}(s^{*}) $ for a torus $ T^{*}\leq M^{*} $ which is a dual group of $ T $.
\end{lemma}

For the rest of the section, we shall present some number theory results.
 
\begin{lemma}\label{outerlessthanconjugacyclasses}
Let $ p $ be a prime and $ f $ a positive integer. Then the following statements hold:
\begin{itemize}
\item[(a)] If $ q=p^{f} >11 $, then $ 6f+1< (q^{2}-q-2)/9 $.
\item[(b)] If $ q=p^{f} \geq7 $ and $ q $ is odd, then $ 4f+1< (q^{2} - 1)/8 $.
\end{itemize}
\end{lemma}

Let $ q,n\geq 2 $ be integers. Suppose that $ (q,n)\neq (2,6) $ and if $ n=2 $ assume that $ q + 1 $ is not a power of $ 2 $.  Then by Zsigmondy's theorem \cite{Zsi92}, a Zsigmondy prime divisor $ l(n) $ always exists. The Zsigmondy prime divisor is defined as a prime $ l(n) $ such that $ l(n)\mid q^{n}-1 $ but $ l(n)\nmid \prod _{i=1}^{n-1}(q^{i}-1) $. 

\begin{lemma}\label{exceptionsforq}
Let $ q =p^{f}$ for some prime $ p $ and a positive integer $ f $. Suppose that $ a$ is a positive integer and  $b$, $ c $ be non-negative integers.
\begin{itemize}
\item[(a)] If $ q - 1=2^{c} $ and $ q + 1 = 2^{a}3^{b} $, then $ q =3, 5 $ or $ 17 $;
\item[(b)] If $ q - 1= 2^{a} $ and $ q + 1=2^{b}5^{c} $, then $ q = 3 $ or $ 9 $;
\item[(c)] If $ q - 1= 2^{a}5^{b} $ and $ q + 1=2^{c} $, then $ q = 3 $.
\end{itemize}
\end{lemma}
\begin{proof}
(a) If $ b=0 $, then $ q = 3 $. Otherwise we have $ 2=2^{a}(3^{b}-2^{c-a}) $, so $ a = 1 $ and $ 3^{b} - 1 = 2^{c-a}$. By Zsigmondy's Theorem, there is a Zsigmondy prime $ l\mid 3^{b} - 1 $ except when $ b\leq 2 $. If $ b=1 $, then $ q = 5 $ and if $ b=2 $, then $ q = 17 $.

(b) If $ c=0 $, then $ q=3 $. If $ c\geq 1 $, then $ a > b $. Now $ 2=2^{b}5^{c}-2^{a}=2^{b}(5^{c}-2^{a-b}) $. Since $ b\geq 1 $, we have that $ b=1 $ and $ 5^{c}-2^{a-b}=1 $, that is, $ 5^{c}- 1= 2^{a-1} $. By Zsigmondy's Theorem \cite{Zsi92}, there exist a Zsigmondy prime $ l\mid 5^{c}-1 $ unless when $ c=1 $. Hence $ q + 1=10 $ and so $ q = 9 $.

(c) If $ b=0 $, then $ q=3 $. If $ b\geq 1 $, then $ 2=2^{c}-2^{a}5^{b}=2^{a}(2^{c-a}-5^{b}) $. Hence $ a=1 $ and $ 2^{c-1}-5^{b}=1 $ which implies that $ 5^{b} + 1= 2^{c-1} $. By \cite[IX, Lemma 2.7]{HB92}, $ b=1 $ which can only happen when $ q=11 $. This is a contradiction since $ q + 1=12\neq 2^{c} $.
\end{proof}

\section{Quasisimple groups with a character vanishing on elements of the same order} \label{vanishingonsameorder}

In this section we prove Theorem \ref{classificationp-powerorder}. In view of \cite[Theorem 1.2]{Mad18q}, it is sufficient to only consider quasisimple groups $ M $ such that $ M/Z(M) $ is isomorphic to a finite simple group of Lie type distinct from $ \PSL_{2}(q) $.
\begin{theorem}\label{Lietypep-power}
Let $ M $ be a quasisimple group such that $ M/Z(M) $ is a finite simple group of Lie type over a field of characteristic $ p $ distinct from $ \PSL_{2}(q) $. If \eqref{e:star} holds for $ M $, then $ M $ is one of the following:
\begin{itemize}
\item[(1)] $ M=\PSU_{3}(4) $, $ \chi(1)=13 $;
\item[(2)] $ M= $ $^{2}\rm{B}_{2}(8)$, $ \chi(1)=14 $.
\end{itemize}
\end{theorem}

\subsection{Classical groups} We shall show that Theorem \ref{Lietypep-power}(1) holds with a series of propositions.

We first show that the Steinberg character of a classical group of Lie type fails to satisfy \eqref{e:star}:
\begin{lemma}{\label{Steinbergclassical}}
Let $ M $ be a finite simple classical group of Lie type over a field of characteristic $ p $, distinct from $ \PSL_{2}(q) $. Then the Steinberg character $ \chi $ of $ M $ fails to satisfy \eqref{e:star}.
\end{lemma}
\begin{proof}
Suppose that $ p=2 $. Then $ \chi $ is of $ 2 $-defect zero and so $ \chi $ vanishes on every $ 2 $-singular element of $ M $. In particular, $ \chi $ vanishes on an involution. By \cite[III, Theorem 5]{Suz61}, $ M $ has an element of order $ 2r $ for some odd prime $ r $ except when $ M\cong \PSL_{3}(4) $. The character table of $ \PSL_{3}(4) $ exhibited in the \Atlas{} \cite{CCNPW85} confirms our conclusion for this special case. We may assume that $ M $ has an element $ g $ of order $ 2r $ with $ r $ as above. Then $ \chi $ vanishes on $ g $ and so vanishes on two elements of distinct orders, contradicting (i) of \eqref{e:star}.

Now we suppose that $ p$ is odd. Then $ \chi $ is of $ p $-defect zero and so $ \chi $ vanishes on every $ p $-singular element of $ M $. In particular, $ \chi $ vanishes on a unipotent element of order $ p $. Now $ M $ has an element $ g $ of order $ pr $, where $ r $ is a prime number, since the size of the connected component of a prime graph of $ M $ containing $ p $ is at least $ 2 $ by \cite[Theorem 1]{Wil81}. Hence $ \chi(g)=0 $ and the result follows.
\end{proof}

\subsubsection{Special Linear Groups}

Let $ \mathcal{M}=\GL_{n}(\mathbb{\overline{F}}_{p}) $ and let $ F $ be the standard Frobenius map. The conjugacy classes of $ F $-stable maximal tori of $ \GL_{n}(\mathbb{\overline{F}}_{p}) $ and $ \SL_{n}(\mathbb{\overline{F}}_{p}) $ are parametrised by conjugacy classes of $ \SSS_{n} $. Recall that conjugacy classes of $ \SSS_{n} $ are parametrised by cycle shapes. If $ \mathcal{T}\leqslant \GL_{n}(\mathbb{\overline{F}}_{p}) $ corresponds to $ \lambda =(\lambda _{1},\lambda _{2},\dots, \lambda _{m}) \in \SSS_{n} $ such that $ \lambda _{1}\geq \lambda _{2}\geq \dots\geq \lambda _{m} $, then $ |T|=|\mathcal{T}^{F}|=\prod _{i=1}^{m}(q^{\lambda _{i}}-1) $ and if $ \mathcal{T}\leqslant \SL_{n}(\mathbb{\overline{F}}_{p}) $, then $ (q-1)|T|=(q-1)|\mathcal{T}^{F}|=\prod _{i=1}^{m}(q^{\lambda _{i}}-1) $. 

\begin{lemma} \cite[Lemmas 3.1 and 4.1]{LM15}\label{regularelementsSLn}
Let $ \lambda \vdash n $ be a partition, and $ \mathcal{T} $ a corresponding $ F $-stable maximal torus of $ \SL_{n}(\mathbb{\overline{F}}_{p}) $ or $ \SU_{n}(\mathbb{\overline{F}}_{p}) $. Assume that either all parts of $ \lambda $ are distinct, or $ q\geq 3 $ and at most two parts of $ \lambda $ are equal. Then $ T= \mathcal{T}^{F} $ contains regular elements. 
\end{lemma}

\begin{lemma}\cite[Lemma 3.2]{LM15}\label{reductivesubgroup=PGL}
Let $ \mathcal{H}\leqslant \PGL_{n}(\mathbb{\overline{F}}_{p}) $ or $ \mathcal{H}\leqslant \PGU_{n}(\mathbb{\overline{F}}_{p}) $ be a reductive subgroup containing $ F $-stable maximal tori corresponding to cycle shapes $ \lambda _{1}, \lambda _{2}, \dots, \lambda _{r} $. If no intransitive or imprimitive subgroup of $ \SSS_{n} $ contains elements of all these cycle shapes, then $ \mathcal{H}= \PGL_{n}(\mathbb{\overline{F}}_{p}) $ or $ \mathcal{H}= \PGU_{n}(\mathbb{\overline{F}}_{p}) $, respectively.
\end{lemma}

To use this result we note that $ \mathcal{M} $ is connected reductive with a Steinberg endomorphism $ F:\mathcal{M}\rightarrow \mathcal{M} $ and $ M:=\mathcal{M}^{F} $. If $ \mathcal{T}^{*}\leqslant \mathbf{C}_{\mathcal{M}^{*}}(s^{*}) $, then since $ \mathcal{T}^{*} $ is connected we have that $ \mathcal{T}^{*}\leqslant \mathbf{C}_{\mathcal{M}^{*}}^{\circ}(s^{*}) $, a reductive subgroup of $ \mathcal{M}^{*} $ (see \cite[Theorem 14.2]{MT11}).

The table below shows Zsigmondy primes $ l_{i} $ for the corresponding tori $ T_{i} $. Note that elements of order $ l_{i} $ in the torus $ T_{i} $ are regular elements. It was shown in \cite{MNO00} that almost all characters of simple groups vanish on elements of order $ l_{1} $ or $ l_{2} $ whenever $ l_{1} $ and $ l_{2} $ exist. 

\begin{center}\label{ZsigmondyTableClassical}
Table 1\\
Tori and Zsigmondy primes for classical groups of Lie type
\end{center}
\begin{center}
\begin{tabular}{|c|c|c|c|c|}
\cline{1-5}
$ M $ & $ |T_{1}| $ & $ |T_{2}| $  & $ l_{1} $ & $ l_{2} $  \\
 \cline{1-5}
$ A_{n} $ & $ (q^{n+1}-1)/(q-1) $ & $ q^{n}-1 $ &  $ l(n+1) $ & $ l(n) $ \\
\cline{1-5}
$ ^{2}A_{n} $ ($ n\geqslant 3 $ odd) & $ (q^{n+1}-1)/(q+1) $ & $ q^{n} + 1 $ &  $ l(n+1) $ & $ l(2n) $ \\
\cline{1-5}
$ ^{2}A_{n} $ ($ n\geqslant 2 $ even) & $ (q^{n+1}+1)/(q+1) $ & $ q^{n}-1 $  & $ l(2n+2) $ & $ l(n) $ \\
\cline{1-5}
$ B_{n}, C_{n} $ ($ n\geqslant 3 $ odd) & $ q^{n} + 1 $ & $ q^{n}-1 $  & $ l(2n) $ & $ l(n) $ \\
\cline{1-5}
$ B_{n}, C_{n} $ ($ n\geqslant 2 $ even) &  $ q^{n} + 1 $ & $ (q^{n-1}+1)(q+1) $ &  $ l(2n) $ & $ l(2n-2) $  \\
\cline{1-5}
$ D_{n} $ ($ n\geqslant 5 $ odd) &  $ q^{n} - 1 $  &   $ q^{n-1}+1)(q+1) $    &  $ l(n) $ & $ l(2n-2) $ \\
\cline{1-5}
$ D_{n} $ ($ n\geqslant 4 $ even) &   $ (q^{n-1}-1)(q-1) $   &  $ (q^{n-1} + 1)(q+1) $   & $ l(n-1) $ & $ l(2n-2) $  \\
\cline{1-5}
$ ^{2}D_{n} $ ($ n\geqslant 4 $) &   $ q^{n} + 1 $   &   $ (q^{n-1} + 1)(q-1) $     & $ l(2n) $ & $ l(2n-2) $ \\
\hline
\end{tabular}
\end{center}

Since $ \PSL_{3}(2)\cong \PSL_{2}(7) $, and $ \PSL_{2}(7) $ is considered in \cite[Theorem 1.2]{Mad18q}, we may assume that $ n = 3 $ and $ q\geq 3 $ for the following result. 
\begin{proposition}\label{sl3}
Let $ M $ be a quasisimple group such that $ M/Z(M)=\PSL_{3}(q) $, where $ q\geq 3 $. Then every non-trivial faithful irreducible character of $ M $ fails to satisfy \eqref{e:star}.
\end{proposition}
\begin{proof} Using explicit character tables in the \Atlas{} \cite{CCNPW85}, we may assume that $ q\geq 13 $. First consider $ Z(M)\neq 1 $. Now, $ |Z(M)|=3 $, $ 3\mid (q-1) $ and by \eqref{e:star}, $ \chi $ vanishes on a $ 3 $-element. Note that unipotent characters are not faithful when $ Z(M)\neq 1 $. Hence we may assume that $ \chi $ is not unipotent. Then $ \chi $ lies in the Lusztig series $ \mathcal{E}(M, s^{*})$ of a semisimple element $ s^{*} $ in the dual group $ M^{*}=\PGL_{3}(q) $. Let $ T_{1} $ and $ T_{2} $ be tori of $ M $ corresponding to the partitions $ (3) $ and $ (2)(1) $, respectively. By Lemma \ref{regularelementsSLn}, the tori $ T_{1} $ and $ T_{2} $ contain regular elements. We claim that $ \chi $ vanishes on regular elements either in $ T_{1} $ or in $ T_{2} $. Otherwise by Lemma \ref{vanishingresultTorus}, $ \textbf{C}_{M^{*}}(s^{*}) $ contains conjugates of the duals $ T_{1}^{*} $ and $ T_{2}^{*} $. This means that the corresponding reductive subgroup $ \mathbf{C}_{\mathcal{M}^{*}}^{\circ}(s^{*}) $ contains $ \mathcal{T}_{1}^{*} $ and $ \mathcal{T}_{2}^{*} $. Using Lemma \ref{reductivesubgroup=PGL}, we have that $ \textbf{C}_{\mathcal{M^{*}}}^{\circ}(s^{*})=\PGL_{3}(\mathbb{\overline{F}}_{p}) $, that is, $ \textbf{C}_{M^{*}}(s^{*})=\PGL_{3}(q) $ and so $ \chi $ is unipotent, contradicting our assumption that $ \chi $ is not unipotent. 
Hence $ \chi $ vanishes on regular elements in $ T_{1} $ or in $ T_{2} $. Suppose that $ \chi $ vanishes on regular elements in $ T_{1} $. Note that $ |T_{1}| $ is divisible by a Zsigmondy prime $ l_{1} $ and $ T_{1} $ contains regular elements of order $ l_{1} $. Since $ \gcd(l_{1},3)=1 $, $ \chi $ vanishes on at least two elements of distinct orders, contradicting \eqref{e:star}. We may thus assume $ \chi $ vanishes on regular elements in $ T_{2} $. If $ q+1 $ is not a power of $ 2 $, then $ |T_{2}| $ is divisible by a Zsigmondy prime $ l_{2} $. By the same argument as above, we may infer $ \chi $ vanishes on at least two elements of distinct orders, contradicting \eqref{e:star}. Suppose $ q + 1 $ is a power of $ 2 $. This means that $ |T_{2}| $ is even and hence $ T_{2} $ contains elements of even order by \cite[Remark 2.2]{LM16}. Hence $ \chi $ also vanishes on an element of even order and the result follows.

Suppose $ M=\PSL_{3}(q) $. Then $ \chi $ is not the Steinberg character by Lemma \ref{Steinbergclassical}. By \cite[Theorem 2.1]{MSW94}, $ \chi $ vanishes on regular elements in $ T_{1} $ or in $ T_{2} $. Suppose that $ \chi $ vanishes on regular elements of $ T_{1} $. Note that $ |T_{1}| $ is divisible by a Zsigmondy prime $ l_{1} $. If $ |T_{1}| $ is divisible by two distinct primes, then the result follows by \cite[Remark 2.2]{LM16}. Suppose that $ |T_{1}| $ is a prime power. Then $ |T_{1}|=(q^{2}+q+1)/\gcd(3,q-1) $ must be prime by \cite{Nag21}. Suppose $ |T_{1}|=\frac{q^{3}-1}{(q-1)\gcd(3,q-1)}=\frac{q^{2}+q+1}{\gcd(3,q-1)}=l_{1} $. Then $ G $ has $ \frac{l_{1}-1}{3}= \frac{q^{2}+q-2}{3\cdot \gcd(3,q-1)} $ conjugacy classes whose elements are of order $ l_{1} $. Now $ |\Out(M)|= 2\cdot\gcd(3,q-1)\cdot f $. By Lemma \ref{outerlessthanconjugacyclasses}, $ |\Out(M)| < 6f + 1\leq \frac{q^{2}+q-2}{9} $ and (ii) of \eqref{e:star} fails to hold.
Suppose $ \chi $ vanishes on regular elements in $ T_{2} $. By \cite[Theorem 2.1]{MSW94}, $ \chi $ vanishes on elements of order $ q + 1 $. If $ q $ is odd, then $ q + 1 $ is even. In particular, $ q + 1 $ is not a prime. By \cite[Theorem 5.1]{MNO00}, $ \chi $ vanishes on an element of prime order which means that $ \chi $ vanishes on two elements of distinct orders, contradicting \eqref{e:star}. Hence we may assume that $ q $ is even so that $ q + 1 $ is odd. We may assume that $ q + 1 $ is prime by the above argument. Since $ |T_{2}|=(q^{2} - 1)/\gcd(3, q - 1) $ and $ (q-1)/\gcd(3,q - 1)\neq 1 $, we have that $ |T_{2}| $ is divisible by at least two primes. Hence there exists a prime $ l $ such that $ l\mid (q - 1) $ which entails the existence of an $ l $-singular regular element in $ |T_{2}| $ by \cite[Remark 2.2]{LM16}. By \cite[Theorem 2.1]{MSW94}, $ \chi $ vanishes on this $ l $-singular element. Hence $ \chi $ vanishes on two elements of distinct orders and the result follows.
\end{proof}
\begin{proposition}\label{sln4non-unipotent}
Suppose that $ M $ is quasisimple such that $ M/Z(M)\cong\PSL_{n}(q) $, $ n\geq 4 $ and $ q\geq 2 $. Then every non-trivial faithful irreducible character of $ M $ fails to satisfy \eqref{e:star}.
\end{proposition}
\begin{proof}
Firstly, suppose that $ n\geq 4 $ and $ q=2 $. For $ M $ isomorphic to $ \PSL_{4}(2) $ or $ \PSL_{5}(2) $ we have explicit character tables in the \Atlas{} \cite{CCNPW85} and for $ M/Z(M) $ isomorphic to $ \PSL_{6}(2) $ or $ \PSL_{7}(2) $, we obtain explicit character tables in \Magma{} \cite{MAGMA}. Hence we may assume that $ n\geq 8 $. Then we have $ 3=q+1 $. Now $ (q+1)^{4}\mid |T| $ for a torus $ T $ corresponding to the partition $ (n-8)(2)(2)(2)(2) $. It follows that $ M $ is of $ 3 $-rank at least $ 4 $. Hence by Theorem \ref{p-rank3}, $ \chi $ vanishes on a $ 3 $-singular element. On the other hand, by \cite[Theorem 2.1]{MSW94}, if $ n $ is even, $ \chi $ vanishes on an element of order $ q^{n/2}+1 $ or an element of order $ q^{n-1}-1 $ and if $ n $ is odd, then $ \chi $ vanishes on an element of order $ q^{n}-1 $ or an element of order $ q^{(n-1)/2}+1 $. Note that in all of the aforementioned cases, the order of elements on which $ \chi $ vanishes, exceeds $ 3 $. Each such order is either relatively prime to $ 3 $ or is $ 3 $-singular. In the former case, $ \chi $ vanishes on an element that is not of prime order. Using \cite[Theorem 5.1]{MNO00}, $ \chi $ vanishes on an element of prime order. Hence $ \chi $ vanishes on at least two elements of distinct orders, contradicting \eqref{e:star}.

Suppose that $ M=\SL_{n}(q) $, $ n\geq 4 $, $ q\geq 3 $ with $ Z(M)\neq 1 $. By \eqref{e:star}, $ |Z(M)| $ is a power of a prime $ l$ that divides $q-1 $ and $ \chi $ necessarily vanishes on an $ l $-element. We claim that $ \chi $ also vanishes on an $ l_{1} $-element or an $ l_{2} $-element. Suppose the contrary. First note that $ \chi $ is not a unipotent character since $ \chi $ is faithful in $ M $. Hence $ \chi $ lies in the Lusztig series $ \mathcal{E}(M,s^{*}) $ of a semisimple element $ s^{*} $ in the dual group $ M^{*}=\PGL_{n}(q) $. Let $ T_{1} $ and $ T_{2} $ denote maximal tori corresponding to the partitions $ (n) $ and $ (n-1)(1) $. Note that $ T_{1} $ and $ T_{2} $ contain regular elements by Lemma \ref{regularelementsSLn}. By Lemma \ref{vanishingresultTorus}, $ \textbf{C}_{M^{*}}(s^{*}) $ contains conjugates of the dual tori $ T^{*}_{1} $ and $ T^{*}_{2} $. The corresponding reductive subgroup $ \textbf{C}_{\mathcal{M}^{*}}^{\circ}(s^{*}) $ contains $ \mathcal{T}^{*}_{1} $ and $ \mathcal{T}^{*}_{2} $. Using Lemma \ref{reductivesubgroup=PGL}, we infer that $ \textbf{C}^{\circ}_{\mathcal{M}^{*}}(s^{*})=\PGL_{n}(\mathbb{\overline{F}}_{p}) $ and so $ s^{*} $ is central. Hence $ s^{*}=1 $ and $ \chi $ is unipotent thus contradicting the assumption that $ \chi $ is not unipotent. Hence our claim is true and the result follows.

Suppose that $ M=\PSL_{n}(q) $, $ n\geq 4 $, $ q\geq 3 $. First suppose that $ n=4 $ and $ q\geq 3 $. We have an explicit character table for $ \PSL_{4}(3) $ in the \Atlas{} \cite{CCNPW85} and for $ \PSL_{4}(4) $ and $ \PSL_{4}(5) $ we obtain an explicit character table in \Magma {}. Assume that $ q\geq 7 $. Note that for $ |T_{1}| $ and $ |T_{2}| $, the Zsigmondy primes $ l_{1} $ and $ l_{2} $ exist, respectively. By the proof of \cite[Theorem 2.1]{MSW94}, $ \chi $ is of $ l_{1} $-defect zero or $ l_{2} $-defect zero and so $ \chi $ vanishes on elements of order $ l_{1} $ or $ l_{2} $. Then $ |T_{1}| =\frac{q^{4}-1}{(q-1)\gcd(4,q-1)}=\frac{(q+1)(q^{2}+1)}{\gcd(4,q-1)} $ is divisible by two distinct primes. Also, $ |T_{2}|=\frac{q^{3}-1}{\gcd(4,q-1)}=\frac{(q-1)(q^{2}+q+1)}{\gcd(4,q-1)} $ is divisible by two distinct primes since $ \frac{q-1}{\gcd(4,q-1)}\neq 1 $. Hence $ \chi $ vanishes on two regular elements of distinct orders. 

Suppose $ n=5 $, $ q\geq 3 $. Assume that $ \chi $ is not unipotent. Let $ T_{1} $, $ T_{2} $ and $ T_{3} $ be tori of $ M $ corresponding to the partitions $ (5) $, $ (4)(1) $ and $ (3)(2) $, respectively. These tori contain regular elements by Lemma \ref{regularelementsSLn}. We claim that $ \chi $ vanishes on regular elements in at least two of these tori. Otherwise, $ \textbf{C}_{M^{*}}(s^{*}) $ contains conjugates of the dual tori $ T^{*}_{i} $ and $ T^{*}_{j} $ of $ T_{i} $ and $ T_{j} $, respectively, $ i\neq j $, $ 1\leq i,j \leq 3 $, where $ \chi $ lies in the Lusztig series $ \mathcal{E}(M,s^{*}) $. The corresponding reductive subgroup $ \textbf{C}_{\mathcal{M}^{*}}^{\circ}(s^{*}) $ contains $ \mathcal{T}^{*}_{i} $ and $ \mathcal{T}^{*}_{j} $. It follows from Lemma \ref{reductivesubgroup=PGL} that $ \textbf{C}^{\circ}_{\mathcal{M}^{*}}(s^{*})=\PGL_{5}(\mathbb{\overline{F}}_{q}) $, that is, $ \chi $ is unipotent, a contradiction. The claim is thus true. Now for $ |T_{1}| $ and $ |T_{2}| $ note that the corresponding Zsigmondy primes $ l_{1} $ and $ l_{2} $ exist, respectively. Hence $ \chi $ vanishes on at least two elements of distinct orders $ l_{1} $, $ l_{2} $ or some positive integer that divides $ |T_{3}| $.

We may assume that $ \chi $ is unipotent. Then $ \chi $ vanishes on elements of order $ l_{1} $ or $ l_{2} $ by the proof of \cite[Theorem 2.1]{MSW94}. It is sufficient to show that $ \chi $ vanishes on an $ l $-singular element with $ l\neq 5 $, an odd prime and $ \gcd(l_{1},l)=\gcd(l_{2},l)=1 $. Let $ q $ be even and note that $ q\geq 3 $. If $ \gcd(5,q - 1)=1 $, then there exists an odd prime $ l\neq 5 $ such that $ l\mid (q-1) $ and $ M $ is of $ l $-rank at least $ 3 $ and $ \gcd(l_{1},l)=\gcd(l_{2},l)=1 $. Hence $ \chi $ vanishes on an $ l $-singular element by Theorem \ref{p-rank3}. If $ \gcd(5,q-1)\neq 1 $, then there exists an odd prime $ l\neq 5 $ such that $ l\mid (q + 1) $ and so $ \gcd(l_{1},l)=\gcd(l_{2},l)=1 $. Note that $ M $ is of $ l $-rank $ 2 $. Then by the proof of \cite[Proposition 3.8]{LM15}, $ \chi $ vanishes on an $ l $-singular element. Assume that $ q $ is odd. Suppose that $ \gcd(5,q-1)=1 $. Then there exists an odd prime $ l\neq 5 $ such that either $ l\mid (q - 1) $ or $ l\mid (q + 1) $, $ \gcd(l_{1},l)=\gcd(l_{2},l)=1 $ and $ M $ is of $ l $-rank $ 2 $ with the following exception: $ q-1=2^{a} $, $ a\geqslant 1 $ and $ q + 1=2^{b}5^{c} $, $ b\geqslant 1 $, $ c\geqslant 0 $. Then by the proof of \cite[Proposition 3.8]{LM15}, $ \chi $ vanishes on an $ l $-singular element for the former case. For the exceptions, $ q=3 $ or $ 9$ by Lemma \ref{exceptionsforq}. If $ q=3 $, then using \Magma{} \cite{MAGMA} to calculate the character table of $ \PSL_{5}(3) $, we conclude that $ \chi $ does not satisfy $ \eqref{e:star} $. Let $ q = 9 $. In this case we look at the orders of $ T_{1} $ and $ T_{2} $. Now $ |T_{1}|=\frac{9^{5}-1}{9-1}=11^{2}\cdot 61 $ and $ |T_{2}|=9^{4}-1=2^{5}\cdot 5\cdot 41 $. Since $ \chi $ is either of $ l_{1} $-defect zero or of $ l_{2} $-defect zero, $ \chi $ vanishes on at least two elements of distinct orders. Assume that $ \gcd(5,q-1)=5 $. If there exists an odd prime $ l\neq 5 $ such that $ l\mid (q - 1) $ or $ l\mid (q + 1) $, then $ M $ is of $ l $-rank at least $ 2 $ and by the proof of \cite[Proposition 3.8]{LM15}, $ \chi $ vanishes on an $ l $-singular element. Hence the only exception we have is when $ q-1=2^{a}5^{b} $ and $ q + 1=2^{c} $. By Lemma \ref{exceptionsforq}, $ q = 3 $ which does not satisfy $ \gcd(5,q-1)=5 $. Hence the result follows.

Suppose that $ n=6 $. Then $ \chi $ vanishes on elements of order $ l_{1} $ or $ l_{2} $ by the proof of \cite[Theorem 2.1]{MSW94}. If $ \gcd(6,q-1)=1 $, then there exist an odd prime $ l\mid (q - 1) $ such that the $ l $-rank of $ M $ is $ 5 $. By Theorem \ref{p-rank3} and since $ \gcd(l_{1},l)=\gcd(l_{2},l)=1 $, it follows that $ \chi $ vanishes on at least two elements of distinct orders.  Let $ \gcd(6,q-1)=2 $. Then $ q $ is odd. If $ q\neq 3 $, then there exist an odd prime $ l $ such that $ l\mid (q - 1) $ or $ l\mid (q + 1) $. In this case $ M $ is of $ l $-rank at least $ 3 $ and we are done. If $ q=3 $, then using \Magma{} \cite{MAGMA} to calculate the character table of $ \PSL_{6}(3) $, we conclude that $ \chi $ does not satisfy \eqref{e:star}. Let $ \gcd(6,q-1)=3 $ or $ 6 $. Then the $ 3 $-rank of $ M $ is $ 4 $ and the result follows.

Suppose that $ n=7 $. Then $ \chi $ vanishes on elements of order $ l_{1} $ or $ l_{2} $ by the proof of \cite[Theorem 2.1]{MSW94}. We first consider $ q $ even. If $ \gcd(7,q-1)=1 $, then there exists an odd prime $ l $ such that $ l\mid (q-1) $ and the $ l $-rank of $ M $ is $ 6 $. If $ \gcd(7,q-1)\neq 1 $, then since $ q $ is even, there exists an odd prime $ l\neq 7 $ such that $ l\mid (q + 1) $ and the $ l $-rank of $ M $ is $ 3 $. Assume that $ q $ is odd. Suppose that $ \gcd(7,q-1)=1 $. Then there exists an odd prime $ l $ such that either $ l\mid (q - 1) $ or $ l\mid (q + 1) $ unless $ q=3 $. If $ q\neq 3 $, then we have an odd prime $ l $ and $ M $ is of $ l $-rank at least $ 3 $.  If $ q=3 $, then using \Magma{} \cite{MAGMA} to calculate the character table of $ \PSL_{7}(3) $, we can conclude that $ \chi $ fails to satisfy \eqref{e:star}. Suppose $ \gcd(7,q-1)=7 $. If $ q + 1 $ is not a power of $ 2 $, then there exists an odd prime $ l $ such that $ l\mid (q + 1) $ and we are done. We may thus assume that $ q + 1=2^{a} $, $ a\geq 3 $. Now $ 3 $ divides either $ q - 1 $, $ q $ or $ q + 1 $. We know that $ 3\nmid (q + 1) $. Suppose that $ 3\mid (q-1) $. Then $ 3 $ is the desired odd prime. Thus $ 3\mid q $, that is, $ q=3^{f} $, $ f\geq 1 $. This implies that $ q=2^{a}-1=3^{f} $. By \cite[IX, Lemma 2.7]{HB92}, $ f=1 $, that is, $ q=3 $, a contradiction since $ \gcd(7, q-1)=7 $.

Suppose that $ n=8 $. Then $ \chi $ vanishes on elements of order $ l_{1} $ or $ l_{2} $ by the proof of \cite[Theorem 2.1]{MSW94}. If there exists an odd prime $ l $ such that $ l\mid (q-1) $, then we are done. We may assume that $ q-1=2^{a} $, $ a\geq 1 $. Then $ q $ is odd. If there exists an odd prime $ l $ such that $ l\mid (q + 1) $, then we are done. Otherwise $ q+1=2^{b} $, $ b\geq 2 $. Then $ q=3 $. For $ M=\PSL_{8}(3) $, $ |T_{1}| $ and $ |T_{2}| $ are both divisible by two distinct primes. Since $ \chi $ is of $ l_{1} $-defect zero or of $ l_{2} $-defect zero, we have that $ \chi $ vanishes on two elements of distinct orders. 

Suppose that $ n\geq 9 $. Then $ \chi $ vanishes on elements of order $ l_{1} $ or $ l_{2} $ by the proof of \cite[Theorem 2.1]{MSW94}. Consider a torus $ T $ of $ M $ corresponding to the partition $ (n-9)(3)(3)(3) $. There exists a Zsigmondy prime $ l $ dividing $ q^{3} - 1 $ such that $ M $ is of $ l $-rank at least $ 3 $. By Theorem \ref{p-rank3}, $ \chi $ vanishes on an $ l $-singular element. Since $ \gcd(l_{1},l)=\gcd(l_{2},l)=1 $, the result follows. This concludes our argument.
\end{proof}

\subsubsection{Special Unitary Groups}

Let $ \mathcal{M}=\GL_{n}(\mathbb{\overline{F}}_{p}) $ and let $ F $ be the twisted Frobenius morphism. The conjugacy classes of $ F $-stable maximal tori of $ \GU_{n}(\mathbb{\overline{F}}_{p}) $ and $ \SU_{n}(\mathbb{\overline{F}}_{p}) $ are also parametrised by conjugacy classes of $ \SSS_{n} $. If $ \mathcal{T}\leqslant \GU_{n}(\mathbb{\overline{F}}_{p}) $ corresponds to the cycle shape $ \lambda=(\lambda _{1},\lambda _{2},\cdots, \lambda _{m}) \in \SSS_{n} $ with $ \lambda _{1}\geq \lambda _{2}\geq \dots\geq \lambda _{m} $, then $ |T|=|\mathcal{T}^{F}|=\prod _{i=1}^{m}(q^{\lambda _{i}}-(-1)^{\lambda _{i}}) $ whilst if $ \mathcal{T}\leqslant \SU_{n}(\mathbb{\overline{F}}_{p}) $, then $ (q+1)|T|=(q+1)|\mathcal{T}^{F}|=\prod _{i=1}^{m}(q^{\lambda _{i}}-(-1)^{\lambda _{i}}) $.

\begin{proposition}\label{SU3}
Let $ M $ be a quasisimple group such that $ M/Z(M)=\PSU_{3}(q) $, $ q\geq 3 $. If \eqref{e:star} holds for $ M $, then $ M=\PSU_{3}(4) $ with $ \chi(1)=13 $. 
\end{proposition}
\begin{proof} We may conclude from the character tables in \Atlas{} \cite{CCNPW85} that $ M=\PSU_{3}(4) $ when $ 3\leq q\leq 11 $. We may assume that $ q\geq 13 $. Note that $ \chi $ is not the Steinberg character.
We first consider the case $ M=\SU_{3}(q) $ and $ Z(M)\neq 1 $. Since we are only considering faithful characters, $ \chi $ is not unipotent. Then $ |Z(M)|=3 $, $ 3\mid (q + 1) $. By (iii) of \eqref{e:star}, $ \chi $ vanishes on a $ 3 $-element. We have that $ T_{1} $ and $ T_{2} $ correspond to the cycle shapes $ (3) $ and $ (2)(1) $ and so $ T_{1} $ and $ T_{2} $ have regular elements by Lemma \ref{regularelementsSLn}. Using the same argument as in Proposition \ref{sl3} we have that $ \chi $ vanishes on regular elements in $ T_{1} $ or in $ T_{2} $. If $ \chi $ vanishes on regular elements in $ T_{1} $, then $ \chi $ vanishes on an element of Zsigmondy prime order $ l_{1} $. Since $ \gcd(l_{1},3)=1 $, the result follows. If $ \chi $ vanishes on regular elements in $ T_{2} $, then $ \chi $ vanishes either on an element of Zsigmondy prime order $ l_{2} $ if $ q-1 $ is not a power of $ 2 $ or on an element of even order if $ q-1 $ is a power of $ 2 $. Since all the orders above are relatively prime to $ 3 $, $ \chi $ vanishes on at least two elements of distinct orders, contradicting \eqref{e:star}.

Let $ M=\PSU_{3}(q) $. By \cite[Lemmas 5.3 and 5.4]{MNO00}, $ \chi $ vanishes on regular elements in $ T_{1} $ or in $ T_{2} $. Assume that $ \chi $ vanishes on regular elements in $ T_{1} $. Note that $ |T_{1}| $ is divisible by a Zsigmondy prime $ l_{1} $. If $ |T_{1}| $ is divisible by two distinct primes, then by \cite[Remark 2.2]{LM16}, $ \chi $ vanishes on at least two elements of distinct orders. Note that $ T_{1} $ is cyclic by \cite[Section 3.3]{Gag73}. If $ |T_{1}|=l_{1}^{a} $, $ a>1 $, then $ \chi $ vanishes on two elements of distinct orders $ l_{1}^{a} $ and $ l_{1} $, which contradicts \eqref{e:star}. We may assume that $ |T_{1}|=\frac{q^{3}+1}{(q+1)\gcd(3,q+1)}=\frac{q^{2}-q+1}{\gcd(3,q+1)}=l_{1} $. If $ q=13 $, then $ M $ has $ 52 $ conjugacy classes of order $ 53 $, $ |\Out(M)|=2 $, contradicting (ii) of \eqref{e:star}. We thus assume $ q\geq 16 $. Then $ M $ has $ \frac{l_{1}-1}{3}\geq \frac{q^{2}-q-2}{3\cdot\gcd(3,q+1)} $ conjugacy classes whose elements are of order $ l_{1} $. Now $ |\Out(M)|\leq 2\cdot \gcd(3,q-1)\cdot f $. By Lemma \ref{outerlessthanconjugacyclasses}, $ 6f + 1\leq \frac{q^{2}-q-2}{9} $ and (ii) of \eqref{e:star} fails to hold. 

We now consider the case where $ \chi $ vanishes on regular elements in $ T_{2} $. By \cite[Theorem 2.2]{MSW94}, $ \chi $ vanishes on an element of order $ q-1 $. On the other hand, $ \chi $ vanishes on: an element of order Zsigmondy prime $ l_{2} $, an involution or on a regular unipotent element by the proof of \cite[Lemma 5.4]{MNO00}. Therefore $ \chi $ vanishes on at least two elements of distinct orders, contradicting \eqref{e:star}.
 \end{proof}
 
\begin{proposition}\label{SU4}
Let $ M $ be a quasisimple group such that $ M/Z(M)\cong \PSU_{n}(q) $, $ n\geq 4 $ and $ q\geq 2 $. Then every non-trivial faithful irreducible character of $ M $ fails to satisfy \eqref{e:star}.
\end{proposition}
\begin{proof}
We consider $ M/Z(M)\cong \PSU_{n}(2) $ first. Using the character tables for $ \PSU_{4}(2) $, $ \PSU_{5}(2) $ and $ \PSU_{6}(2) $ in \Atlas{} \cite{CCNPW85}, and for $ \PSU_{7}(2) $, $ \PSU_{8}(2) $ and $ \PSU_{9}(2) $ from \Magma{} \cite{MAGMA}, we may assume that $ n\geq 10 $. Suppose that $ Z(M)\neq 1 $. This means that $ |Z(M)|=3 $ and $ 3\mid (q + 1) $. Note that $ \chi $ is not unipotent. By $ (\star) $, $ \chi $ vanishes on a $ 3 $-element. We claim that $ \chi $ vanishes on regular elements in $ T_{1} $ or in $ T_{2} $. Assume that this claim is not true. Then $ \textbf{C}_{M^{*}}(s^{*}) $ contains conjugates of the dual tori $ T_{1}^{*} $ and $ T_{2}^{*} $ of $ T_{1} $ and $ T_{2} $ where $ \chi $ lies in the Lusztig series $ \mathcal{E}(M,s^{*}) $. The corresponding reductive subgroup $ \textbf{C}_{\mathcal{M}^{*}}^{\circ}(s^{*}) $ contains the tori $ \mathcal{T}^{*}_{1} $ and $ \mathcal{T}^{*}_{2} $. By Lemma \ref{reductivesubgroup=PGL}, $ \textbf{C}_{\mathcal{M}^{*}}^{\circ}(s^{*})=\PGL_{n}(\mathbb{\overline{F}}_{2}) $ and $ s^{*} $ is central. Hence $ \textbf{C}_{M^{*}}(s^{*})=\PGU_{n}(2) $, and so $ \chi $ is unipotent, a contradiction. The claim is true and $ \chi $ vanishes either on an $ l_{1} $-element or on an $ l_{2} $-element. Hence $ \chi $ vanishes on at least two elements of distinct orders. 

Assume that $ Z(M)=1 $. By \cite[Theorem 2.2]{MSW94}, $ \chi $ vanishes on elements of order $ l_{1} $ or $ l_{2} $. Consider a torus $ T $ of $ M $ corresponding to the partition $ (n-10)(2)(2)(2)(2)(2) $. Hence $ M $ is of $ l $-rank at least $ 3 $, where $ l= q + 1=3 $. By Theorem \ref{p-rank3}, $ \chi $ vanishes on an $ l $-singular element. Hence the result follows.

Suppose that $ M/Z(M)\cong \PSU_{n}(q) $, $ n\geq 4 $, $ q\geq 3 $. Assume that $ Z(M)\neq 1 $. By $ (\star) $, $ |Z(M)| $ is a power of a prime $ l\mid (q+1) $ and $ \chi $ vanishes on an $ l $-element. Using the proof of \cite[Theorem 2.2]{MSW94}, $ \chi $ vanishes on an $ l_{1} $-element or an $ l_{2} $-element and the result follows. 

Suppose that $ M\cong \PSU_{n}(q) $. By the proof of \cite[Theorem 2.2]{MSW94}, $ \chi $ is of $ l_{1} $-defect zero or $ l_{2} $-defect zero. Suppose $ n\geq 9 $ and consider a torus $ T $ of $ M $ corresponding to the partition $ (n-9)(3)(3)(3) $. Then there exists a Zsigmondy prime $ l=l(6) $ dividing $ q^{3} + 1 $ and $ M $ is of $ l $-rank at least $ 3 $. By Theorem \ref{p-rank3}, $ \chi $ vanishes on an $ l $-singular element and the result follows. Hence we may assume that $ n\leq 8 $.

Suppose that $ n=8 $. Recall that $ q\geq 3 $. If $ \gcd(8,q + 1)=1 $, then there exists an odd prime $ l\mid (q + 1) $ such that the $ l $-rank of $ M $ is $ 7 $. By Theorem \ref{p-rank3}, $ \chi $ vanishes on an $ l $-singular element and the result follows since $ \gcd(l_{1},l)=\gcd(l_{2},l)=1 $. If $ \gcd(8,q + 1)\neq 1 $, then $ q $ is odd and there exists an odd prime $ l $ such that $ l\mid (q-1) $ or $ l\mid (q+1) $ unless $ q=3 $. If $ q\neq 3 $, then $ M $ is of $ l $-rank at least $ 3 $ and hence $ \chi $ vanishes on an $ l $-singular element for an odd prime $ l $ by Theorem \ref{p-rank3}. If $ q=3 $, then using \Magma{} \cite{MAGMA} to calculate the character table of $ \PSU_{8}(3) $, we can conclude that $ \chi $ fails to satisfy \eqref{e:star}. Hence the result follows.

Suppose that $ n=7 $. We first consider the case when $ q $ is even. If we have $ \gcd(7,q+1)=1 $, then there exists an odd prime $ l\neq 7 $ such that $ l\mid (q+1) $. If $ \gcd(7,q+1)\neq 1 $, then since $ q $ is even, there exists an odd prime $ l\neq 7 $ such that $ l\mid (q - 1) $. In both cases, $ M $ is of $ l $-rank at least $ 3 $ and so $ \chi $ vanishes on an $ l $-singular element. Since $ \gcd(l_{1},l)=\gcd(l_{2},l)=1 $, $ \chi $ vanishes on at least two elements of distinct orders. Assume that $ q $ is odd. Suppose that $ \gcd(7,q+1)=1 $. Then there exists an odd prime $ l\neq 7 $ such that $ l\mid (q - 1) $ or $ l\mid (q + 1) $ unless $ q=3 $. If $ q\neq 3 $, then we have an odd prime $ l $ and $ M $ is of $ l $-rank at least $ 3 $ which means that $ \chi $ vanishes on at least two elements of distinct orders. If $ q=3 $, then using \Magma{} \cite{MAGMA} to calculate the character table of $ \PSU_{7}(3) $, we conclude that $ \chi $ does not satisfy $ (\star) $. Suppose $ \gcd(7,q+1)=7 $. If $ q - 1 $ is not a power of $ 2 $, then there exists an odd prime $ l\neq 7 $ such that $ l\mid (q - 1) $. Hence $ M $ is of $ l $-rank more than $ 3 $ and $ \gcd(l_{1},l)=\gcd(l_{1},l)=1 $. Thus $ \chi $ vanishes on two elements of distinct orders, a contradiction to \eqref{e:star}.

We may assume that $ q - 1=2^{a} $, $ a\geq 3 $. Now $ 3 $ divides either $ q - 1 $, $ q $ or $ q + 1 $. We know that $ 3\nmid (q - 1) $. Suppose that $ 3\mid (q+1) $. Then $ 3 $ is the desired odd prime since $ M $ is of $ 3 $-rank at least $ 3 $. Thus $ 3\mid q $, that is, $ q=3^{f} $, $ f\geq 1 $. This implies that $ q-1 = 3^{f}-1 = 2^{a}$. By Zsigmondy's Theorem, there is a Zsigmondy prime $ l\mid (3^{f}-1) $ unless $ f\leq 2 $. If $ f=1 $, then $ q=3 $, contradicting the hypothesis that $ \gcd(7, q+1)=7 $. If $ f=2 $, then $ q=9 $, again contradicting the hypothesis that $ \gcd(7, q+1)=7 $.

Suppose that $ n=6 $. If $ \gcd(6, q + 1)=1 $, then there exist an odd prime $ l\mid q + 1 $ such that the $ l $-rank of $ M $ is $ 5 $. Let $ \gcd(6, q + 1)=2 $. Then $ q $ is odd. If $ q \neq 3 $, then there exist an odd prime $ l $ such that $ l\mid (q - 1) $ or $ l\mid (q + 1) $ and the result follows since $ M $ is of $ l $-rank at least $ 3 $. Let $ \gcd(6, q + 1)=3 $. Then $ q $ is even. Note that $ q\geq 4 $. Then there exist an odd prime $ l $  that divides $ q -1 $. Hence $ \chi $ vanishes on an $ l $-singular element since the $ l $-rank of $ M $ is $ 6 $. Let $ \gcd(6, q + 1)=6 $. Then $ q $ is odd. If there exists an odd prime $ l\neq 3 $ such that $ l\mid (q + 1) $, then the result follows. We may assume that $ q + 1=2^{a}3^{b} $, $ a\geq 1 $ and $ b\geq 1 $. Then there exists an odd prime $ l\neq 3 $ that divides $ q - 1 $ and the result follows unless $ q - 1=2^{c} $, $ c\geq 3 $. Hence we may assume that $ q - 1=2^{c} $. By Lemma \ref{exceptionsforq}, $ q = 5 $ or $ 17 $ since $ \gcd(6, q + 1)=6 $. In both cases, the $ 3 $-rank of $ M $ is $ 5 $ and by Theorem \ref{p-rank3}. Hence, result follows.  

Suppose $ n=5 $. Assume that $ \chi $ is not unipotent. Let $ T_{1} $, $ T_{2} $ and $ T_{3} $ be tori of $ M $ corresponding to $ (5) $, $ (4)(1) $ and $ (3)(2) $, respectively. These tori contain regular elements by Lemma \ref{regularelementsSLn}. We claim that $ \chi $ vanishes on regular elements in at least two of these tori. Otherwise, $ \textbf{C}_{M^{*}}(s^{*}) $ contains conjugates of the dual tori $ T^{*}_{i} $ and $ T^{*}_{j} $ of $ T_{i} $ and $ T_{j} $, respectively, $ i\neq j $, $ 1\leq i,j \leq 3 $, where $ \chi $ lies in the Lusztig series $ \mathcal{E}(M,s^{*}) $. The corresponding reductive subgroup $ \textbf{C}_{\mathcal{M}^{*}}^{\circ}(s^{*}) $ contains $ \mathcal{T}^{*}_{i} $ and $ \mathcal{T}^{*}_{j} $. It follows from Lemma \ref{reductivesubgroup=PGL} that $ \textbf{C}^{\circ}_{\mathcal{M}^{*}}(s^{*})=\PGU_{5}(\mathbb{\overline{F}}_{q}) $, that is, $ \chi $ is unipotent, a contradiction. The claim is thus true. Now for $ |T_{1}| $ and $ |T_{2}| $ note that the corresponding Zsigmondy primes $ l_{1} $ and $ l_{2} $ exist, respectively. Hence $ \chi $ vanishes on at least two elements of distinct orders $ l_{1} $, $ l_{2} $ or some positive integer that divides $ |T_{3}| $.

Assume that $ \chi $ is unipotent. Then $ \chi $ vanishes on elements of order $ l_{1} $ or $ l_{2} $ by the proof of \cite[Theorem 2.1]{MSW94}. By Theorem \ref{p-rank3}, it is sufficient to show that $ \chi $ vanishes on an $ l $-singular element with $ l\neq 5 $, an odd prime. Let $ q $ be even and note that $ q\geq 3 $. If $ \gcd(5,q + 1)=1 $, then there exists an odd prime $ l\neq 5 $ such that $ l\mid (q+1) $ and $ M $ is of $ l $-rank at least $ 3 $. If $ \gcd(5,q+1)\neq 1 $, then there exists an odd prime $ l\neq 5 $ such that $ l\mid (q -1) $. Note that $ M $ is of $ l $-rank $ 2 $. By the proof of \cite[Proposition 4.2]{LM15}, $ \chi $ vanishes on an $ l $-singular element. Now assume that $ q $ is odd. Suppose that $ \gcd(5,q+1)=1 $. Then there exists an odd prime $ l\neq 5 $ such that $ l\mid (q - 1) $ or $ l\mid (q + 1) $ with the following exception: $ q-1=2^{a} $, $ a\geqslant 1 $ and $ q + 1=2^{b}5^{c} $, $ b\geqslant 1 $, $ c\geqslant 0 $. By Lemma \ref{exceptionsforq}, $ q=3 $ or $ 9$. In both cases, using \Magma{} \cite{MAGMA} to calculate the character tables of $ \PSU_{5}(3) $ and $ \PSU_{5}(9) $, we conclude that $ \chi $ does not satisfy (\ref{e:star}). Assume that $ \gcd(5,q+1)=5 $. If there exists an odd prime $ l\neq 5 $ such that $ l\mid (q - 1) $ or $ l\mid (q + 1) $, then $ M $ is of $ l $-rank at least $ 2 $ and we are done by \cite[Proposition 4.2]{LM15}. Hence the only exception we have is when $ q-1=2^{a}5^{b} $ and $ q + 1=2^{c} $. Lemma \ref{exceptionsforq} entails $ q = 3 $ which contradicts the assumption that $ \gcd(5,q+1)=5 $. 
%

First suppose that $ n=4 $ and $ q\geq 3 $. We have an explicit character table for $ \PSU_{4}(3) $ in the \Atlas{} \cite{CCNPW85} and for $ \PSU_{4}(4) $ and $ \PSU_{4}(5) $ we obtain an explicit character table in \Magma {}. Assume that $ q\geq 7 $. Note that for $ |T_{1}| $ and $ |T_{2}| $, the Zsigmondy primes $ l_{1} $ and $ l_{2} $ exist, respectively. By the proof of \cite[Theorem 2.1]{MSW94}, $ \chi $ is of $ l_{1} $-defect zero or $ l_{2} $-defect zero and so $ \chi $ vanishes on elements of order $ l_{1} $ or $ l_{2} $. Then $ |T_{1}| =\frac{q^{4}-1}{(q+1)\gcd(4,q+1)}=\frac{(q-1)(q^{2}+1)}{\gcd(4,q+1)} $ is divisible by two distinct primes. Also, $ |T_{2}|=\frac{q^{3}+1}{\gcd(4,q+1)}=\frac{(q+1)(q^{2}-q+1)}{\gcd(4,q+1)} $ is divisible by two distinct primes since $ \frac{q+1}{\gcd(4,q+1)}\neq 1 $. Hence $ \chi $ vanishes on two regular elements of distinct orders. 
\end{proof}

\subsubsection{Symplectic Groups and Special Orthogonal Groups}

Let $ \mathcal{M} $ be a simple, simply connected algebraic group of type $ B_{n} $, $ C_{n} $ or $ D_{n} $ over $ \mathbb{\overline{F}}_{p} $ and let $ F:\mathcal{M} \rightarrow \mathcal{M} $ be a Frobenius morphism such that $ M:=\mathcal{M}^{F} $. Then the $ \mathcal{M}^{F} $-conjugacy classes of $ F $-stable maximal tori of $ \mathcal{M} $ are parametrised by the conjugacy classes of $ W $, the Weyl group of $ \mathcal{M} $. If $ \mathcal{M} $ is of type $ B_{n} $ or $ C_{n} $, then $ W $ is isomorphic to the wreath product $ C_{2}\wr S_{n} $ and the conjugacy classes of $ W $ are parametrised by pairs of partitions $ (\lambda, \mu)\vdash n $. (see \cite[Section 2.1]{LM16} for details). In particular, if a maximal torus $ T=\mathcal{T}^{F} $ corresponds to a partition $ (\lambda, \mu)=((\lambda_{1},\lambda_{2}, ...,\lambda _{r}),(\mu _{1}, \mu _{2}, ..., \mu _{s}))\vdash n $, then 
\begin{center}
$ |T|=\prod_{i=1}^{r}(q^{\lambda_{i}}-1)\prod_{j=1}^{s}(q^{\mu_{j}}+1) $
\end{center}
and $ \mathcal{T}^{F} $ contains cyclic subgroups of orders $ q^{\lambda_{i}}-1 $ and $ q^{\mu_{j}}+1 $ for all $ i $ and $ j $.

If $ \mathcal{M} $ is of type $ \mathrm{D}_{n} $, then $ W=C_{2}^{n-1}\rtimes \SSS_{n} $ and the $ \mathcal{M}^{F} $-conjugacy classes of $ F $-stable maximal tori of $ \mathcal{M} $ are parametrised by pairs of partitions $ (\lambda, \mu)\vdash n $ such that $ \mu $ has an even number of parts if $ \mathcal{M}^{F} $ is $ \mathrm{Spin}^{+}_{2n}(q) $ and $ \mu $ has an odd number of parts if $ \mathcal{M}^{F} $ is the non-split orthogonal group $ \mathrm{Spin}^{-}_{2n}(q) $. Now $ |T| $ is the same as in the case when $ \mathcal{M} $ is of type $ \mathrm{B}_{n} $ or $ \mathrm{C}_{n} $, that is,
\begin{center}
$ |T|=\prod_{i=1}^{r}(q^{\lambda_{i}}-1)\prod_{j=1}^{s}(q^{\mu_{j}}+1) $.
\end{center}

\begin{lemma}\cite[Lemma 2.1]{LM16}\label{regularelementsBCandD}
Let $ \mathcal{M} $ be a simple, simply connected classical group of type $ B_{n} $, $ C_{n} $ or $ D_{n} $ defined over $ \mathbb{\overline{F}}_{p} $ with corresponding Steinberg morphism  $ F $.

Let $ (\lambda, \mu)=((\lambda_{1},\lambda_{2},...,\lambda _{r}),(\mu _{1}, \mu_{2},...,\mu_{s})) $ be a pair of partitions of $ n $, and $ \mathcal{T} $ a corresponding $ F $-stable maximal torus of $ \mathcal{M} $. Then $ T=\mathcal{T}^{F} $ contains regular elements if one of the following is fulfilled:
 
 \begin{itemize}
 \item[(1)] $ q > 3 $, $ \lambda_{1} < \lambda _{2} < \cdots < \lambda _{r} $ and $ \mu _{1} < \mu _{2} < \cdots < \mu_{s} $;
 \item[(2)] $ q\in \{2,3\}$, $\lambda_{1}<\lambda_{2}< ...< \lambda _{r} $, $ \mu _{1}<\mu _{2}< ... < \mu _{s} $, all $ \lambda_{i}\not= 2 $, and if $ \mathcal{M} $ is of type $ B_{n} $ or $ C_{n} $, then also all $ \lambda_{i}\not= 1 $; or
 \item[(3)] $ \mathcal{M} $ is of type $ D_{n} $, $ 2 < \lambda_{1} < \lambda_{2} < \cdots < \lambda_{r} $ and $ 1=\mu_{1}=\mu_{2} < \mu_{3} < \cdots < \mu_{s} $.
 \end{itemize}
\end{lemma}

\begin{lemma}\cite[Lemma 2.3]{LM16}\label{reductivesubgroup=dualBCD}
Let $ \mathcal{M} $ be a simple algebraic group of type $ B_{n} $, $ C_{n} $(with $ n\geq 2 $) or $ D_{n} $(with $ n\geq 4 $) with Frobenius endomorphism  $ F $ such that $ \mathcal{M}^{F} $ is a classical group. Let $ \Lambda $ be a set of pairs of partitions $ (\lambda, \mu)\vdash n $. Assume the following:
\begin{itemize}
\item[(1)] there is no $ 1\leq k \leq n-1 $ such that all $ (\lambda, \mu)\in \Lambda $ are of the form $ (\lambda _{1},\mu _{1}) \sqcup (\lambda _{2}, \mu _{2}) $ with $ (\lambda _{1},\mu _{1})\vdash k $;
\item[(2)] the greatest common divisor of all parts of all $ (\lambda, \mu)\vdash n $ is $ 1 $; and
\item[(3)] if $ \mathcal{M} $ is of type $ B_{n} $, then there exist pairs $ (\lambda, \mu)\vdash n $ for which $ \mu $ has an odd number of parts, and one for which $ \mu $ has an even number of parts.
\end{itemize}
If $ s\in \mathcal{M}^{F} $ is semisimple such that $ \textbf{C}_{\mathcal{M}}(s) $ contains maximal tori of $ \mathcal{M} $ corresponding to all $ (\lambda, \mu )\in \Lambda $, then $ s $ is central.
\end{lemma}

We first consider a quasisimple group $ M $ such that  $ M/Z(M)\cong \mathrm {PSp}_{4}(q) $. Since $ \rm{Sp}_{4}(2)'\cong \PSL_{2}(9) $, $ \mathrm{PSp}_{4}(3)\cong \PSU_{4}(2) $, and the groups $ \PSL_{2}(9) $ and $ \PSU_{4}(2) $ were dealt with in \cite[Theorem 1.2]{Mad18q} and Proposition \ref{SU4}, respectively, we shall not consider them in the result below.
\begin{proposition}
Let $ M $ be a quasisimple group such that $ M/Z(M)\cong \mathrm{PSp}_{4}(q) $, where $ q\geq 4 $. Then every non-trivial faithful irreducible character of $ M $ fails to satisfy \eqref{e:star}
\end{proposition}
\begin{proof} Since the character tables $ \rm{Sp}_{4}(4) $ and $ \rm{Sp}_{4}(5) $ are in \Atlas{} \cite{CCNPW85}, we may assume that $ q\geq 7 $. Suppose that $ q $ is even. Then the result follows from the generic character tables in \Chevie{} \cite{CHEVIE}. We may assume that $ q $ is odd, $ q\geq 7 $. For this case we first suppose $ Z(M)\neq 1 $. Note that $ \chi $ is not unipotent. Then $ |Z(M)|=2 $ and  by \eqref{e:star}, $ \chi $ vanishes on a $ 2 $-element. For each prime $ l $ such that $ l\mid (q - 1)  $ or $ l\mid (q + 1) $, the Sylow $ l $-subgroups of $ G $ are non-cyclic. Since $ q \neq 3 $, there exists an odd prime $ l $ such that $ l\mid (q -1)  $ or $ l\mid (q + 1) $. By \cite[Theorem 4.1]{LM16}, $ \chi $ vanishes on an $ l $-singular element. But $ \gcd(2,l)=1 $, so $ \chi $ vanishes on at least two elements of distinct orders, as required.

Suppose $ M\cong \PSp_{4}(q) $. We may assume that $ \chi $ is not the Steinberg character. By the proof of \cite[Theorem 2.3]{MSW94}, $ \chi $ vanishes on regular elements in $ T_{1} $ or in $ T_{2} $. In particular, we may choose two conjugacy classes $ \mathcal{C}_{1} $ and $ \mathcal{C}_{2} $ in $ T_{1} $ and $ T_{2} $ such that $ \chi $ vanishes on $ \mathcal{C}_{1} $ or $ \mathcal{C}_{2} $. Now $ \mathcal{C}_{2} $ may contain elements which are not of Zsigmondy prime order. In that case the result follows since $ \chi $ vanishes on elements of prime order by \cite[Theorem 5.1]{MNO00}. Hence we may assume that $ \mathcal{C}_{1} $ and $ \mathcal{C}_{2} $ contain elements of Zsigmondy prime orders. Suppose that $ \chi $ vanishes on elements in  $ T_{2} $. Note that $ |T_{2}|=(q^{2}-1)/2 $ is even. Hence $ T_{2} $ contains a regular element of even order by \cite[Remark 2.2]{LM16} and $ \chi $ vanishes on this element. This means that $ \chi $ vanishes on two elements of distinct orders, contradicting \eqref{e:star}. Assume that $ \chi $ vanishes on elements of $ T_{1} $. Note that $ T_{1} $ is cyclic by \cite[Section 4.5]{Gag73}. If $ |T_{1}| $ is not prime, then there exist at least two elements of distinct orders on which $ \chi $ vanishes. We may assume that $ |T_{1}|=\frac{q^{2}+1}{2} $  is prime. Then there are $ \dfrac{\frac{(q^{2} + 1)}{2}-1}{4}=\frac{q^{2}-1}{8} $ conjugacy classes with elements of order $ \frac{q^{2}+1}{2} $. On the other hand, we have that $ |\Out(M/Z(M)|\leq 4f $, where $ q=p^{f} $, $ p $ is a prime and $ f\geq 1 $. By Lemma \ref{outerlessthanconjugacyclasses}, $ 4f + 1 < \frac{q^{2}-1}{8} $ and the result follows by (ii) of \eqref{e:star}. 
\end{proof}

Let $ \mathcal{S}=\{ \mathrm{PSp}_{2n}(q)\mid n\geq 3\}\cup \{ \mathrm{PSO}_{2n+1}(q)\mid n\geq 3\} \cup \{\mathrm{PSO}^{\pm}_{2n}(q)\mid n\geq 4\} $.

\begin{proposition}
Let $ M $ be a quasisimple group such that $ M/Z(M)\in \mathcal{S} $. Then every non-linear faithful irreducible character $ \chi $ of $ M $ fails to satisfy \eqref{e:star}.
\end{proposition}

\begin{proof} Note that $ \chi $ is not the Steinberg character by Lemma \ref{Steinbergclassical}. We first consider the case where $ M\in \mathcal{S} $ with $ q=2 $. Since we have character tables in \Atlas{} \cite{CCNPW85} for $ \mathrm{Sp}_{2n}(2) \cong\mathrm{SO}_{2n+1}(2)$, $ 3\leq n\leq 4 $ and $ \mathrm{PSO}^{\pm}_{2n}(2) $, $ 4 \leq n \leq 5 $, we may assume that $ n\geq 5 $ and $ n\geq 6 $, respectively. Since $ q + 1=3 $, then $ M $ is of $ 3 $-rank at least $ 5 $. By Theorem \ref{p-rank3}, $ \chi $ vanishes on a $ 3 $-singular element. For $ M\cong \mathrm{Sp}_{2n}(2) $, $ \chi $ vanishes on elements of order $ l_{1} $ or elements of order $ l_{2} $ in Table \ref{ZsigmondyTableClassical}, or $ \chi $ is of $ l_{3} $-defect zero, where $ l_{3}=l(n-1) $ (the last case only arising when $ n $ is even) by \cite[Lemmas 5.3-5.5]{MNO00}. Note that Zsigmondy primes $ l_{1} $, $ l_{2} $, $ l_{3} $ exist, and we have $ \gcd(l_{1}, 3)=\gcd(l_{2}, 3)=\gcd(l_{3}, 3)=1 $. Hence $ \chi $ vanishes on at least two elements of distinct orders, contradicting \eqref{e:star}. For $ M\cong  \mathrm{PSO}^{\pm}_{2n}(2) $, $ n\geq 6 $, $ \chi $ vanishes on elements of order $ l_{1} $, or $ l_{2} $ or $ \chi $ is of $ l_{3} $-defect zero, where $ l_{3}=l(2n-4) $ (the last case only arising when $ n $ is even) by \cite[Lemmas 5.3, 5.4 and 5.6]{MNO00}. Since the Zsigmondy primes $ l_{1} $, $ l_{2} $ and $ l_{3} $ exist, the result follows.

Henceforth we may assume that $ q \geq 3 $ and $ n\geq 3 $. Suppose that $ Z(M)\neq 1 $. Then $ \gcd(2,q-1)=2 $ and by \eqref{e:star}, $ \chi $ vanishes on a $ 2 $-element. We want to show that $ \chi $ also vanishes on an element of Zsigmondy prime order. Note that $ \chi $ is not unipotent and so $ \chi $ lies in the Lusztig series $ \mathcal{E}(M,s^{*}) $ of $ s^{*} $ in the dual $ M^{*} $.
Let $ ((\lambda),(\mu)) $ and $ ((\lambda'),(\mu'))) $ the partitions corresponding to tori $ T_{1} $ and $ T_{2} $ with orders in Table 1. These tori contain regular elements by Lemma \ref{regularelementsBCandD}. We claim that $ \chi $ vanishes on regular elements in at least one of these tori. Otherwise by Lemma \ref{vanishingresultTorus}, $ \textbf{C}_{M^{*}}(s^{*}) $ contains conjugates of the dual tori $ T^{*}_{1} $ and $ T^{*}_{2} $. The corresponding subgroup $ \mathbf{C}_{\mathcal{M}^{*}}(s^{*}) $ contains conjugates of the dual tori $ \mathcal{T}^{*}_{1} $ and $ \mathcal{T}^{*}_{2} $.  It follows from Lemma \ref{reductivesubgroup=dualBCD} that $ s^{*} $ is central. Hence $ \textbf{C}_{M^{*}}(s^{*})=M^{*} $, that is, $ \chi $ is unipotent, a contradiction. The claim is thus true. Now for $ T_{1} $ and $ T_{2} $ note that the Zsigmondy primes $ l_{1} $ and $ l_{2} $ exist in respect of $ |T_{1}| $ and $ |T_{2}| $. Hence $ \chi $ vanishes on at least two elements of distinct orders and we are done. 

Suppose that $ Z(M)=1 $. Consider $ M \cong \mathrm{PSp}_{2n}(q)$, $n\geq 3 $, or $\mathrm{PSO}_{2n+1}(q)$, $n\geq 3$. By \cite[Lemmas 5.3-5.5]{MNO00}, $ \chi $ vanishes on elements of order $ l_{1} $, $ l_{2} $ or $ \chi $ is of $ l_{3} $-defect zero, where $ l_{3}=l(n-1) $ (the last case arising when $ n $ is even). In all cases the Zsigmondy primes exist. Now there exists an odd prime $ l $ such that $ l\mid (q - 1) $ or $ l\mid (q + 1) $ except when $ q =3 $. Note that $ M $ is of $ l $-rank at least $ 3 $. If $ q \neq 3 $, then by Theorem \ref{p-rank3}, $ \chi $ vanishes on an $ l $-singular element. Since $ \gcd(l_{1},l)=\gcd(l_{2}, l)=\gcd(l_{3},l)=1 $, the result follows. We are left with case when $ q=3 $. If $ n\geq 6 $, then $ M $ has a torus $ T $ corresponding to $ (-,(n-6)(2)(2)(2)) $, i.e. $ M $ is of $ l $-rank at least $ 3 $, where $ l\mid (q^{2} + 1) $. The result follows again. Hence we may assume that $ n \leq 5 $, that is, $ M\in \{ \rm{PSp}_{6}(3) $, $ \rm{PSp}_{8}(3) $, $ \rm{PSp}_{10}(3) $, $ \rm{PSO}_{7}(3) $, $ \rm{PSO}_{9}(3) $, $ \rm{PSO}_{11}(3)\} $. We have explicit character tables for $ \rm{PSp}_{6}(3) $ and $ \rm{PSO}_{7}(3) $ in the \Atlas{} \cite{CCNPW85}, and using \Magma{} \cite{MAGMA} for the rest of the groups, we have our conclusion.

Suppose that $ Z(M)=1 $ and $ M\cong \mathrm{PSO}^{-}_{2n}(q)  $ with $ n\geq 4 $ and $ q\geq 3 $. By the proof of \cite[Theorem 2.5]{MSW94}, $ \chi $ is of $ l_{1} $-defect zero or of $ l_{2} $-defect zero. If $ q\neq 3 $, then there exists an odd prime $ l $ such that $ l\mid (q - 1) $ or $ l\mid (q + 1) $ and $ M $ is of $ l $-rank at least $ 3 $. The result then follows by Theorem \ref{p-rank3} and since $ \gcd(l_{1},l)=\gcd(l_{2},l)=1 $. Consider $ n\geq 4 $ and $ q = 3 $. We have an explicit character table for $ \mathrm{PSO}^{-}_{8}(3) $ in the \Atlas{} \cite{CCNPW85} and for $ \mathrm{PSO}^{-}_{10}(3) $ we obtain an explicit character table in \Magma {}. For $ n=6 $, the orders of $ T_{1} $ and $ T_{2} $ are divisible by two distinct primes and the result follows. We may assume that $ n\geq 7 $. Hence $ M $ is of $ l $-rank at least $ 3 $ when $ l\mid q^{2} + 1 $. Therefore, by Theorem \ref{p-rank3}, $ \chi $ vanishes on at least two elements of distinct orders.

Suppose that $ Z(M)=1 $ and $ M\cong \mathrm{PSO}^{+}_{2n}(q)  $ with $ n\geq 4 $ and $ q\geq 3 $. Assume that $ n $ is odd. Then the Zsigmondy primes $ l_{1} $ and $ l_{2} $ exist, and $ \chi $ vanishes on regular elements in $ T_{1} $ or in $ T_{2} $ by \cite[Lemma 5.3]{MNO00}. If $ q\neq 3 $, then there exists an odd prime $ l $ such that $ l\mid (q - 1) $ or $ l\mid (q + 1) $ and $ M $ is of $ l $-rank at least $ 3 $. Hence $ \chi $ vanishes on an $ l $-singular element by Theorem \ref{p-rank3}, and thus $ \chi $ vanishes on at least two elements of distinct orders. Let $ q = 3 $. If $ n\geq 7 $, then consider a torus corresponding to the cycle shape $ (-,(n-6)(2)(2)(2)) $. It follows that $ M $ is of $ l $-rank at least $ 3 $ with $ l\mid (q^{2} + 1) $ and by Theorem \ref{p-rank3}, $ \chi $ vanishes on an $ l $-singular element. Hence we may assume $ n\leq 5 $. Hence $ n=5 $ and so $ M\cong \mathrm{PSO}^{+}_{10}(3) $. Using \Magma{} \cite{MAGMA}, the result follows. 

Suppose that $ n\geq 4 $ is even. By \cite[Lemma 5.6]{MNO00}, $ \chi $ vanishes on regular elements of order $ l_{1} $ or $ l_{2} $ or $ \chi $ is of $ l_{3} $-defect zero where $ l_{3}=l(2n-4) $. An argument similar to that used above allows us to dispose of the case when $ q\neq3 $. Suppose $ q= 3 $. Now if $ n\geq 8 $, then $ M $ is of $ l $-rank at least $ 3 $ where $ l $ is an odd prime dividing $ q^{2} + 1 $. In particular, $ l=5 $. By Theorem \ref{p-rank3}, $ \chi $ vanishes on a $ 5 $-singular element. Since $ \gcd(l_{1},5)=\gcd(l_{2},5)=\gcd(l_{3},5)=1 $, the result follows. If $ n=4 $, that is, $ M\cong \mathrm{PSO}^{+}_{8}(3) $, then we have the explicit character table in the \Atlas{} \cite{CCNPW85} and if $ n=6 $, we obtain an explicit character table in \Magma {} for $ \mathrm{PSO}^{+}_{12}(3) $. This concludes our proof.
\end{proof}

\subsection{Exceptional groups}

\subsubsection{Exceptional groups of small Lie rank}
Since $ \PSL_{2}(8)\cong $ $^{2}\mathrm{G}_{2}(3)' $ and $ \PSU_{3}(3)\cong \mathrm{G}_{2}(2)' $, and $ \PSL_{2}(8) $, $ \PSU_{3}(3) $ were dealt with in \cite[Theorem 1.2]{Mad18q} and Proposition \ref{SU3}, respectively, we exclude them in this section.

Let $ \mathcal{L}=$  $\{^{2}\mathrm{B}_{2}(q^{2})\mid q^{2}=2^{2f+1}, f\geq 1\}\cup \{^{2}\mathrm{G}_{2}(q^{2})\mid q^{2}=3^{2f+1},f\geq 1\}\cup \{ ^{2}\mathrm{F}_{4}(q^{2})\mid q^{2}>2\}\cup \{ \mathrm{G}_{2}(q)\mid q\geq 3
\} \cup \{ ^{3}\mathrm{D}_{4}(q)\mid q\geq 2 \} $.

\begin{proposition}
Let $ M $ be a quasisimple group such that $ M/Z(M)\in \mathcal{L}\cup \{^{2}\mathrm{F}_{4}(2)'\} $. If \eqref{e:star} holds for $ M $, then $ M= $  $^{2}\mathrm{B}_{2}(8)$ with $ \chi(1)=14 $.
\end{proposition}
\begin{proof}
The simple group $ M= $  $^{2}\rm{B}_{2}(8) $ satisfies the conclusion of our proposition from its character table in the \Atlas{} \cite{CCNPW85}. For the rest of the groups, using explicit character tables in \Atlas{} \cite{CCNPW85} and generic ordinary character tables in \Chevie{} \cite{CHEVIE}, we may conclude that every non-trivial character of $ M $ does not satisfy conditions of \eqref{e:star}. Hence the result follows. 
\end{proof}

\subsubsection{Exceptional groups of large Lie rank}

The table below shows the Zsigmondy primes $ l_{i} $ for the corresponding tori $ T_{i} $. It was shown in \cite{MNO00} that every non-trivial irreducible character which is not the Steinberg character, vanishes on an element of order $ l_{i} $ for some $ i=1, 2, 3 $.

\begin{center}\label{ZsigmondyTable}
Table 2\\
Tori and Zsigmondy primes for groups of Lie type
\end{center}
\begin{center}
\begin{tabular}{|c|c|c|c|c|c|c|}
\cline{1-7}
$ M $ & $ |T_{1}| $ & $ |T_{2}| $  & $ |T_{3}| $ & $ l_{1} $ & $ l_{2} $ & $ l_{3} $  \\
\cline{1-7}
$ F_{4}(q) $  &   $ \Phi_{12} $   &   $ \Phi _{8} $    &       & $ l(12) $ & $ l(8) $ & \\
\cline{1-7}
$ E_{6}(q) $  &   $ \Phi_{12}\Phi_{3} $   &   $ \Phi _{9} $    &  $\Phi_{8}\Phi_{2}\Phi_{1}$     & $ l(12) $ & $ l(9) $ & $ l(8) $\\
\cline{1-7}
$ ^{2}E_{6}(q) $  &   $ \Phi_{18} $   &   $ \Phi _{12}\Phi _{6} $    &  $\Phi_{8}\Phi_{2}\Phi_{1}$     & $ l(18) $ & $ l(12) $ & $ l(8) $\\
\cline{1-7}
$ E_{7}(q) $  &   $ \Phi_{18}\Phi_{2} $   &   $ \Phi _{14}\Phi _{2} $    &  $\Phi_{12}\Phi_{3}\Phi_{1}$     & $ l(18) $ & $ l(14) $ & $ l(12) $\\
\cline{1-7}
$ E_{8}(q) $  &   $ \Phi_{30} $   &   $ \Phi _{24} $   &  $\Phi_{20}$  & $ l(30) $ & $ l(24) $ & $ l(20) $\\
\hline
\end{tabular}
\end{center}

\begin{proposition}
Let $ M $ be a quasisimple exceptional finite group of Lie type over a field of characteristic $ p $ and of rank at least $ 4 $. Then every non-trivial faithful irreducible character of $ M $ fails to satisfy \eqref{e:star}.
\end{proposition}
\begin{proof}
Note that the group $ M $ must be one of these type: $ \rm{F}_{4} $, $ \rm{E}_{6} $, $ ^{2}\rm{E_{6}} $, $ \rm{E_{7}} $ or $ \rm{E}_{8} $.

Suppose that $ M/Z(M)\cong \mathrm{F}_{4}(2) $. Then using the character tables in the \Atlas{} \cite{CCNPW85}, the result follows. Let $ M=\mathrm{F}_{4}(q) $, $ q\geq 3 $. Then $ M $ is simple. From \cite[Lemma 5.9]{MNO00} we have that $ \chi $ vanishes on regular elements of order $ l_{1} $, $ l_{2} $ or $ l_{3}=l(3) $, or $ \chi $ is of $ p $-defect zero if $ \chi $ is the Steinberg character. On the other hand, there exists an odd prime $ l $ such that $ l\mid (q-1) $ or $ l\mid (q+1) $ and $ M $ is of $ l $-rank at least $ 3 $ unless $ q=3 $. If $ q\neq 3 $, then by Theorem \ref{p-rank3}, we have that $ \chi $ vanishes on an $ l $-singular element and since $ \gcd(l_{1},l)=\gcd(l_{2},l)=\gcd(l_{3},l)=\gcd(p,l)=1 $, the result follows. If $ q = 3 $, then using an explicit character table of $ M=\mathrm{F}_{4}(3) $ from \Magma{} \cite{MAGMA}, the result follows. 

Now suppose $ M=\mathrm{E_{6}}(q)$, $ q\geq 2 $ with $ Z(M)=1 $. From \cite[Lemma 5.9]{MNO00} we have that $ \chi $ vanishes on regular elements of order $ l_{1} $, $ l_{2} $ or $ l_{3} $, or $ \chi $ is of $ p $-defect zero if $ \chi $ is the Steinberg character. On the other hand, $ M $ is of $ l $-rank at least $ 3 $ for an odd prime $ l $ such that $ l\mid (q^{3} - 1) $. By Theorem \ref{p-rank3}, $ \chi $ vanishes on an $ l $-singular element. Since $ \gcd(l_{1}, l)=\gcd(l_{2},l)=\gcd(l_{3},l)=\gcd(p,l)=1 $, we are done. Now suppose that $ Z(M)\neq 1 $, that is, $ |Z(M)|=3 $. By \eqref{e:star}, $ \chi $ vanishes on a $ 3 $-element. Using the above argument, $ \chi $ vanishes on an $ l $-singular element, where $ l\neq3 $ is an odd prime such that $ l\mid (q^{2} + q + 1) $. Since $ \gcd(3,l)=1 $, the result follows.

Suppose that $ M=\mathrm{^{2}E_{6}}(q)$, $ q\geq 2 $ with $ Z(M)=1 $. By \cite[Lemma 5.9]{MNO00}, we have that $ \chi $ vanishes on regular elements of order $ l_{1} $, $ l_{2} $ or $ l_{3} $, or $ \chi $ is of $ p $-defect zero if $ \chi $ is the Steinberg character. On the other hand, $ M $ is of $ l $-rank at least $ 3 $ for an odd prime $ l\neq 3 $ such that $ l\mid (q^{2} - q + 1) $. By Theorem \ref{p-rank3}, $ \chi $ vanishes on an $ l $-singular element. Since $ \gcd(l_{1}, 1)=\gcd(l_{2},l)=\gcd(l_{3},l)=\gcd(p,l)=1 $, the result follows. Now suppose that $ Z(M)\neq 1 $, so $ |Z(M)|=3 $ and $ q=3b-1 $ for some positive integer $ b\geq 2 $ using the \Atlas{} \cite{CCNPW85} for $ \rm{^{2}E_{6}}(2) $. We may assume that $ q\geq 5 $.  It is sufficient to show that $ M $ is of $ l $-rank at least $ 3 $ for some prime $ l\neq 3 $. The candidate for $ l $ is  an odd prime $ l $ such that $ l\mid (q^{2} - q + 1) $. Finally let $ M/Z(M)=\mathrm{^{2}E}_{6}(2) $ with $ |Z(M)|=2 $ since by \eqref{e:star}(iii), $ |Z(M)| $ is a prime power. Then the character table in the \Atlas{} \cite{CCNPW85} concludes this case.

We will use the same arguments for $ M=\mathrm{E}_{7}(q) $ and $ M=\mathrm{E}_{8}(q) $ in the case when $ M $ is simple. By \cite[Lemma 5.9]{MNO00} we have that $ \chi $ vanishes either on regular elements of order $ l_{1} $, $ l_{2} $ or $ l_{3} $, or $ \chi $ is of $ p $-defect zero if $ \chi $ is the Steinberg character. By Theorem \ref{p-rank3}, $ \chi $ vanishes on an $ l $-singular element, where $ l $ is an odd prime such that $ l\mid (q^{2} - q + 1) $. Hence the result follows.

 Now suppose $ M=\mathrm{E}_{7}(q) $ and $ Z(M)\neq 1 $. By \eqref{e:star}, $ \chi $ vanishes on a $ 2 $-element. Using Theorem \ref{p-rank3}, $ \chi $ vanishes on an $ l $-singular element, where $ l $ is an odd prime such that $ l\mid (q^{3} + 1) $. Hence the result follows.
\end{proof}

\section{Non-solvable groups with a character vanishing on one class} 

\subsection{Almost simple groups of Lie type}

We are in a position to prove Theorem \ref{classificationoneclass}.

\begin{proof}[{\textbf{Proof of Theorem \ref{classificationoneclass}}}] 
Suppose that $ \chi \in \Irr (G) $ is faithful, primitive and vanishes on exactly one conjugacy class. By \cite[Theorem 3.3]{Mad18q}, there exist normal subgroups $ M $ and $ Z $ of $ G $ such that $ G/Z $ is an almost simple group and $ M $ is a quasisimple group. In particular, it was shown in \cite{Mad18q} that the normal subgroup $ M $ necessarily satisfies \eqref{e:star}. For $ M/Z $ isomorphic to a sporadic simple group, alternating group $ \A_{n} $, $ n\geq 5 $ or $ \PSL_{2}(q) $, $ q\geq 4 $, the result follows from \cite[Theorems 1.2 and 1.3]{Mad18q}. For $ M/Z $ isomorphic to a finite group of Lie type distinct from $ \PSL_{2}(q) $, $ q\geq 4 $, the result follows from our results in Section \ref{vanishingonsameorder}.

Conversely suppose that one of (1)-(7) in the hypothesis holds. Then for (1)-(6), the result follows from \cite[Theorem 1.3]{Mad18q}. Now when $ G=$  $^{2}\rm{B_{2}}(8){:}3 $ with $ \chi(1)=14 $, by the same argument above \cite[Theorem 5.2]{Mad18q}, the primitivity of $ \chi $ follows.
\end{proof}

\section*{Acknowledgements}
This work will form part of the author's PhD thesis, and he would like to thank his advisors, Professors Tong-Viet and van den Berg, and Dr. Le, for their helpful suggestions. He would also like to thank Professor Malle for the careful reading of an earlier version of this article and pointing out some gaps in the proofs.

\end{document}